\documentclass[12pt,leqno,,a4paper]{amsart}
\usepackage{amssymb, amsmath, enumerate,}
\usepackage{enumerate,color}

\newtheorem{theorem}{Theorem}[section]
\newtheorem{lemma}[theorem]{Lemma}

\newtheorem*{thmA}{Theorem A}
\newtheorem*{thmB}{Corollary B}

\theoremstyle{remark}
\newtheorem{remark}[theorem]{Remark}

\numberwithin{equation}{section}

\begin{document}
\title[indices of non-supersolvable maximal subgroups   in finite groups]{indices  of non-supersolvable maximal subgroups in finite groups}
    \author[Beltr\'an and Shao]{Antonio Beltr\'an\\
     Departamento de Matem\'aticas\\
      Universitat Jaume I \\
     12071 Castell\'on\\
      Spain\\
     \\Changguo Shao \\
College of Science\\ Nanjing University of Posts and Telecommunications\\
     Nanjing 210023 Yadong\\
      China\\
     }

 \thanks{Antonio Beltr\'an: abeltran@uji.es ORCID ID: https://orcid.org/0000-0001-6570-201X \newline
 \indent Changguo Shao: shaoguozi@163.com ORCID ID: https://orcid.org/0000-0002-3865-0573}

\keywords{Maximal subgroups; Supersolvable groups, Prime power index subgroups, Simple groups}

\subjclass[2010]{20E28, 20D15, 20D06}

\begin{abstract} Two classic results, due to K. Doerk and P. Hall respectively, establish the solvability of those finite groups all of whose maximal subgroups are supersolvable,
and the solvability  of finite groups in which all maximal subgroups have prime or squared prime index.
In this note we describe the structure of the  non-solvable finite groups whose  maximal subgroups are either supersolvable or  have prime or squared prime index.

\end{abstract}

\maketitle

\section{Introduction}

It is an established fact that the structure of a  finite group is influenced, or even  determined, by the properties of some or all of its maximal subgroups.
 In recent years, several authors have obtained distinct solvability criteria and remarkable properties for the groups whose maximal subgroups satisfy certain conditions relating to solvability or nilpotency (see  \cite{Guo, LPZ, SLT, YJK}).
   Similarly, in a recent investigation \cite{BS}, the authors demonstrated that if each maximal subgroup of a group $G$ is nilpotent or has prime or the square of a prime index,
    then $G$ is solvable. By replacing nilpotency by supersolvability, we aim to investigate further the structure of those finite groups whose non-supersolvable  maximal subgroups
     have prime or squared prime index.

      \medskip
      The research presented here is founded on two fundamental results: the solvability of minimal non-supersolvable groups, which were initially studied by K. Doerk \cite{Doerk},
      and a classic theorem of P. Hall \cite[VI.9.4]{Hup}, which states that if the index of each maximal
subgroup of a finite group $G$ is a prime or the square of a prime,  then $G$ is
solvable.
 It is noteworthy that the groups that verify our hypotheses do not necessarily  have to be solvable.
 The smallest non-abelian simple group, the alternating group ${\rm Alt}(5)$, serves as an illustrative example.
   However,  the structure of the non-solvable groups satisfying our assumptions is  subject to severe restrictions, as detailed in our main result.
    The reader is also referred to \cite{Ballester, Ballester2}, where several structural properties  and a solvability criterion are obtained in relation to normality and the number of non-supersolvable subgroups, respectively.

\begin{thmA}
Let $G$ be a non-solvable finite group and suppose that every  maximal subgroup of   $G$ is either supersolvable or has prime or squared prime index.
Let  $S(G)$ denote the solvable radical of $G$.  Then $S(G)$ is supersolvable, and one of the following possibilities holds:

\begin{itemize}

\item[(i)]  $G/S(G)$ is non-abelian simple and is isomorphic to ${\rm PSL}_2(5)$, ${\rm PSL}_2(7)$, ${\rm PSL}_2(8)$ or ${\rm PSL}_2(11)$;

\item[(ii)] $G/S(G)\cong {\rm Sym}(5)$, ${\rm PGL}_2(7)$ or ${\rm P\Gamma L}_2(8)$;

\item[(iii)]$S(G)<{\bf O}^2(G)<G$ and  ${\bf O}^{2}(G)/S(G)\cong  S_1\times \cdots \times  S_n,$
 with $n\geq 1$,  where all $S_i$ are isomorphic to  ${\rm PSL}_2(p^{2^a})$, for some odd prime $p$ and $a\geq 1$, or for some odd prime $p\equiv \pm 1$ $(mod~ 8)$  and $a=0$.
\end{itemize}

\end{thmA}

\medskip
In general, the product of two normal supersolvable subgroups of a finite group need not be supersolvable,
 and  this leads to the conclusion that the supersolvable radical of a group does not exist.
 Nevertheless, we want to note that the  groups satisfying the assumptions of Theorem A do possess a largest normal supersolvable subgroup, which coincides with the solvable radical.
An immediate particular case of Theorem A is the following.

\begin{thmB}
There exist exactly four  non-abelian simple groups whose non-supersolvable maximal subgroups have prime or squared prime index.
 These are ${\rm PSL}_2(5)$, ${\rm PSL}_2(7)$, ${\rm PSL}_2(8)$ and ${\rm PSL}_2(11)$.
\end{thmB}

The strategy employed to prove Theorem A is based on the use of Guralnick's classification of non-abelian simple groups that possess a subgroup with prime power index,
 as well as  a variation of this theorem due to Demina and Maslova concerning groups whose non-solvable maximal subgroups have prime power index.
 Of course, both results depend on the Classification of Simple Finite Groups. Similarly, we have drawn upon a recent classification of non-abelian simple groups
  that have a subgroup of squared prime index \cite[Theorem C]{BS}, and other auxiliary results,
   including a characterization of supersolvability for the wreath product of more elementary character \cite{Durbin}.

\medskip
In this paper, all groups are assumed to be finite, and we follow standard notation (e.g.
 \cite{Robinson}). For further details regarding notation, properties and subgroup structure of finite simple groups, the reader is directed to various sources \cite{Con,Hup, Liebeck, Liebeck2}.
\medskip

\section{Preliminaries}

\medskip

As mentioned in the Introduction, the first result presented here is  Guralnick's classification of non-abelian simple groups having a subgroup of prime power index.
 It is based on the Classification  of Finite Simple Groups and is a crucial element of our development.

\begin{theorem} {\normalfont  \cite[Theorem 1]{Gura}} \label{g}
\ Let $G$ be a non-abelian simple group with $H < G$ and
$|G : H| = p^a$, $p$ prime. One of the following holds.

\begin{itemize}

\item[$(a)$] $G={\rm Alt}(n)$, and $H\cong {\rm Alt}(n-1)$, with $n= p^a$.

\item[$(b)$] $G = {\rm PSL}_n(q)$ and $H$ is the stabilizer of a line or hyperplane. Then
$|G: H|= (q^n-1)/(q - 1) = p^a$. (Note that $n$ must be prime).

\item[$(c)$] $G={\rm PSL}_2(11)$ and $H\cong {\rm Alt}(5)$.

\item[$(d)$] $G=M_{23}$ and $H\cong M_{22}$ or $G=M_{11}$ and $H\cong M_{10}$.

\item[$(e)$] $G= {\rm PSU}_4(2)\cong {\rm PSp}_4(3)$ and $H$ is the parabolic subgroup of
index $27$.
\end{itemize}
\end{theorem}

The subsequent result, also referenced in the Introduction, is a slight variation of Guralnick's classification. It is essential  in the proof of Theorem A.

\begin{theorem}{\label{2}} \cite[Theorem 1]{DM}  Let $G$ be a non-solvable group in which every non-solvable maximal subgroup has
prime power index. Then:
\begin{itemize}

\item[(i)] the non-abelian composition factors of the group $G$ are pairwise isomorphic and are exhausted
by groups from the following list:
 \begin{itemize}
\item[(1)] ${\rm PSL}_2(2^p)$, where $p$ is a prime,
\item[(2)] ${\rm PSL}_2(3^p)$, where $p$ is a prime,
\item[(3)]  ${\rm PSL}_2(p^{2^a})$, where $p$ is an odd prime and $a \geq 0$,

\item[(4)] ${\rm Sz}(2^p)$, where $p$ is an odd prime,

\item[(5)] ${\rm PSL}_3(3)$;
\end{itemize}
\item[(ii)] for any simple group $S$ from the list in statement $({\rm i})$, there exists a group $G$ such that any
of its non-solvable maximal subgroups has primary index and $Soc(G)\cong S$.
\end{itemize}
\end{theorem}

In addition, we need another specific case of Guralnick's classification.

 \begin{theorem} \label{clas} \cite[Theorem C]{BS} Let $S$ be a finite non-abelian simple group that has a subgroup of index $q^2$ with $q$ prime. Then one of the following holds.
\begin{itemize}
\item[(a)] $S \cong {\rm Alt}(q^2)$ with subgroups of index $q^2$ $(q\geq 3)$, or
\item[(b)] $S\cong {\rm PSL}_2(8)$ with subgroups of index $3^2$, or
\item[(c)] $S\cong {\rm PSL}_5(3)$ with  subgroups of index $11^2$.
\end{itemize}
\end{theorem}

For the reader's convenience, in the next lemma,  we  compile in detail the structure of the  Sylow normalizers of ${\rm PSL}_2(q)$ with $q$ a prime power.

\begin{lemma}\label{lemma2}  Let $G={\rm PSL}_2(q)$, where $q$ is a power of prime $p$ and $d = (2, q + 1)$. Let $r$ be a prime divisor of $|G|$ and $R\in {\rm Syl}_r(G)$.

\begin{itemize}

\item[(1)] If $r=p$, then ${\bf N}_G(R)=R\rtimes C_{\frac{q-1}{d}}$;

\item[(2)] If $2\neq r\mid \frac{q+1}{d}$, then ${\bf N}_G(R)=C_{\frac{q+1}{d}}\rtimes C_2$;

\item[(3)] If $2\neq r\mid \frac{q-1}{d}$, then ${\bf N}_G(R)=C_{\frac{q-1}{d}}\rtimes C_2$;

\item[(4)] Assume $p\neq r=2$.

\begin{itemize}

\item[(4.1)] If $q\equiv\pm1$ $(mod~ 8)$, then ${\bf N}_G(R)=R$;

\item[(4.2)] If $q\equiv\pm3$ $(mod~ 8)$, then ${\bf N}_G(R)={\rm Alt}(4)$.

\end{itemize}
\end{itemize}
\end{lemma}

\begin{proof} This follows from \cite[Theorem 2.8.27]{Hup}.
\end{proof}

The following result, related to non-solvable minimal normal subgroups supplemented by maximal subgroups, is also needed in our development.

\begin{lemma}\label{m2}\cite[Lemma 2.4]{Q} Let $G$ be a finite group and $N$ a minimal normal subgroup of $G$ such that $N=S_1\times \cdots\times S_t$,
 with $ t\geq 1$ and $S_i$ are isomorphic non-abelian simple groups. Assume that $H$ is a maximal
subgroup of $G$ such that $G= HN$. Set $K=H\cap N$. Then one of the following holds:
\begin{itemize}
\item[(1)] $K=K_1\times \cdots\times K_t$, where $K_i< S_i$ for all $i$, and $H$ acts transitively on the
set $\{K_1,\cdots, K_t\}$.

\item[(2)] $K=A_1\times \cdots\times A_k$ is minimal normal in $H$, where $A_1\cong \cdots\cong A_k\cong S_1$, $k\mid t$ and $k<t$,
  and every nonidentity element of $A_i$ has length $t/k$ under the decomposition.
  \end{itemize}
\end{lemma}

 As stated in the Introduction, we will employ the following characterization of supersolvability for standard (restricted) wreath products.

\begin{theorem} \label{wr} \cite[Theorem]{Durbin} A finite (standard) wreath product $A \wr B$ is supersolvable if and
only if:
\begin{itemize}
\item[(i)] $A$ is nilpotent,
\item[(ii)] either $B$ is abelian or $B'$ is a (nontrivial) $p$-group with $A$ a
$p$-group for the same prime $p$, and
\item[(iii)] for each prime $q$ dividing $|A|$, $ q\equiv -1$ $ ($mod $m)$, where $m$ is the
exponent of $B/Q$, $Q$ being the Sylow $q$-subgroup of $B$ ($Q$ is unique by
(ii), and may be trivial).
\end{itemize}
\end{theorem}

 Finally, we state a  property related to maximal subgroups of direct products as well as some consequences of it, which will be useful when dealing with wreath products in the proof of Theorem A. Its proof is elementary, so we omit it.

 \begin{lemma} \label{d} Let $G$ be a finite group, $D$ the diagonal subgroup of $G\times G$ and $K$ a normal subgroup of $G$.  Then 
$H=\{(x,y)\mid x^{-1}y\in K\}$  $(*)$
is a subgroup of $G\times G$ that contains $D$. Inversely,
if $D \leq H \leq G\times G$, then there exists a normal subgroup $K$
 in $G$, defined as $K=\{g\in G \mid (g,1)\in H\}$ such that constructs $H$ by  formula $(*)$.
 \end{lemma}
 
\begin{remark} \label{rd} As a result of Lemma \ref{d}, it  turns out that  there exists a one-to-one correspondence between the  subgroups of $G\times G$ containing the diagonal subgroup and the  normal subgroups of $G$. Another consequence of Lemma \ref{d} is that a finite group $G$ is simple if and only if the diagonal subgroup is maximal in $G\times G$.
\end{remark}

\section{Proof}

\begin{proof}[Proof of Theorem A]
Let  $\overline{G}:=G/S(G)> 1$. Assume first that every maximal subgroup of $\overline{G}$ is supersolvable.
Since $\overline{G}$ cannot be supersolvable because $G$ is non-solvable, then $\overline{G}$ must be a minimal non-supersolvable group, and in turn, is solvable by \cite[Proposition 3(a)]{Doerk},
  a contradiction.
  Therefore, we can assume that $\overline{G}$ has a maximal subgroup $\overline{M_1}$ that is not supersolvable, so $M_1$ has prime or squared prime index in $G$ by hypothesis.
  On the other hand, since $\overline{G}$ is not solvable, not every maximal subgroup of $\overline{G}$ has prime or squared prime index in $\overline{G}$; otherwise, $\overline{G}$ would be solvable according to a well-known result of Hall (quoted in the Introduction).
   Thus, if $\overline{M_2}$ is one of such subgroups, then, again by hypothesis, $M_2$ must be supersolvable.
   We conclude that a maximal subgroup with prime or squared prime index, and a supersolvable maximal subgroup containing $S(G)$ both exist in $G$.
In particular, $S(G)$ is supersolvable, thereby proving the first assertion of the theorem. Henceforth,  we will assume $S(G)=1$.

\medskip
  We claim that $G$ possesses a unique minimal normal subgroup, say $N$.
  Suppose, on the contrary, that there is another minimal normal subgroup $K$ of $G$. Since $S(G)=1$, obviously $K$ is non-solvable.
  Now,   take $M$ to be a supersolvable  maximal subgroup of $G$, which we know to exist.
   The non-solvability of $N$ implies that $G=NM$. As $K\cap N=1$ then
 $$K\cong KN/N\leq G/N\cong M/(N\cap M),$$
 from which we deduce that $K$ is supersolvable, a contradiction. Thus, the claim is proved.

\medskip

   We can write $N=S_1\times \cdots \times S_n$, where $S_i$ are isomorphic  non-abelian simple groups and $n\geq 1$.
   Let $M$ be a non-solvable maximal subgroup of $G$, which obviously is not supersolvable. Then the hypotheses imply that $M$ has prime or squared prime index in $G$.
    In particular, $G$ satisfies the conditions of Theorem \ref{2}, so we conclude that $S_i$ (for every $i$) is isomorphic to one of the groups listed in the thesis of that theorem.
             Hence, we will need later to appeal to the subgroup structure of certain simple groups, particularly ${\rm PSL}_2(p^f)$, for which we refer the reader to \cite[II.8.27]{Hup}.
 On the other hand, the fact that ${\bf C}_G(N)=1$ leads to $N\unlhd G\leq {\rm Aut}(N)$ (we identify $N$ and Inn$(N)$). These properties will also be taken into account throughout  the proof.
  In the following, we distinguish  two possibilities: $N=G$ or $N<G$. We will see that the former  gives rise to (i), while  the latter yields to (ii) and (iii).

\medskip
 {\bf Case 1}. Let us assume that $N=G$, or equivalently, $G$ is non-abelian simple. By hypothesis, there are two possibilities to consider: $G$ possesses a subgroup of prime index or  squared prime index.

 \medskip

 (1.1) First assume that $G$ has a subgroup of index $q^2 $ with $q$ a prime.
  In view of Theorem \ref{clas} and  the fact that $G$ belongs to the list appearing  in Theorem \ref{2}, we  conclude  that $G\cong {\rm PSL}_2(8)$.
   This group meets the hypotheses of the theorem, so this provides one  of the cases of (i).

\medskip
 (1.2) Suppose that $G$ has a prime index subgroup. Then $G$ satisfies the hypotheses of Theorems \ref{g} and \ref{2}, so we may combine both to obtain the following.
  First, (4)  of Theorem \ref{2}  is certainly discarded, so $G$ must be isomorphic to ${\rm PSL}_3(3)$, ${\rm PSL}_2(2^p)$, ${\rm PSL}_2(3^p)$
   with $p$ prime or  ${\rm PSL}_2(p^{2^a})$ with $p$ odd. On the other hand, according to Theorem \ref{g},
    $G\cong {\rm PSL}_2(4) \cong {\rm Alt}(5)$, $G\cong {\rm PSL}_2(11)$ or $G\cong {\rm PSL}_n(q_0)$ with $q_0$ a prime power such that $(q_0^n-1)/(q_0-1)$ is a prime power too.
     By joining both sets of possibilities, and taking into account that the only isomorphisms among distinct projective special linear groups are
      ${\rm PSL}_2(4)\cong {\rm PSL}_2(5)$ and ${\rm PSL}_3(2)\cong {\rm PSL}_2(7)$, it follows that necessarily $n=2$, with the exceptions of the cases when $n=3$ and either $q_0=2$ or $q_0=3$. We discuss below these possibilities.

\medskip
   Assume first that $n=2$.  Then, by Theorem \ref{g}(b), we know that ${\rm PSL}_2(2^p)$ has a subgroup of prime index  if and only if $(2^{2p}-1)/(2^p -1)=2^p +1$ is prime.
    However,  this can only occur when $p=2$, because if $p$ is an odd prime, then the fact that $2^p\equiv (-1)^p ({\rm mod}$ $3)$ implies that $2^p +1$ is divisible by $3$. This yields to  $G\cong {\rm PSL}_2(5)$.
     Similarly, as $3^p +1$  is an even number when $p$ is an odd prime, it cannot be a prime, whence we deduce that $G$ cannot be isomorphic to ${\rm PSL}_2(3^p)$.
      For  $G={\rm PSL}_2(p^{2^a})$ with $p$ odd, apart from ${\rm PSL}_2(11)$ which corresponds to case (c) of Theorem \ref{g} and has a subgroup of index $11$, according to  Theorem \ref{g}(b), we have that  $(p^{2^{2a}}-1)/(p^{2^a}-1)=p^{2^a}+1$ with $a\geq 0$ should be a prime.  But $p$ is odd, so the even number $p^{2^a}+1$ cannot be prime, and thus, these groups can be dismissed too.
       For the remaining cases, that is, when $n=3$,  one easily checks in \cite{Con} that ${\rm PSL}_3(3)$ possesses non-supersolvable maximal subgroups (isomorphic to ${\rm Sym}(4$)) whose index is not prime nor squared prime, so this group is discarded too. Finally, if $q_0=2$, then $G\cong {\rm PSL}_3(2)\cong  {\rm PSL}_2(7)$, which satisfies the required  conditions in the theorem. As a result, we conclude that $G\cong {\rm Alt}(5)$, ${\rm PSL}_2(7)$ or  ${\rm PSL}_2(11)$, all of which satisfy the hypotheses.
          This result, when considered in conjunction with (1.1), yields (i).

 \medskip
{\bf  Case 2}. Assume $N<G$. We distinguish two subcases depending on the supersolvability of the maximal subgroups that do not contain $N$.
 Note that these subgroups necessarily exist for $N$ being non-nilpotent. Indeed, if every maximal subgroup of $G$ contained  $N$, then $N$ would lie in the Frattini subgroup of $G$, a contradiction.

 \medskip
 (2.1) Suppose that there exists a maximal subgroup $M$ such that $N\not\leq M$ and $M$ is not supersolvable.
  According to the hypotheses we distinguish:  $|G:M|$ is a prime or a squared prime.

  \medskip
 Assume first that $|G:M|=q$ is prime. The uniqueness of $N$ implies that core$_G(M)=1$, so $G\leq {\rm Sym}(q)$.
  Additionally, the maximality of $M$ gives $NM=G$, and hence $|N:N\cap M|=q$.
  Then $q$ is the largest prime dividing both $|G|$ and $|N|$, and the $q$-part of $|G|$ is $|G|_q=q$. As a consequence, $n=1$, or equivalently, $N$ is simple.
   Recall the $N$ is one of the simple groups listed in Theorem \ref{2}, so by combining this with Theorem \ref{g} we get the same conclusion by reasoning as in (1.2):
     $N\cong {\rm Alt}(5)$, ${\rm PSL}_2(7)$, ${\rm PSL}_2(11)$, or ${\rm PSL}_3(3)$, which have subgroups of index $5,7,11$ and $13$ respectively,.  Next, we examine  each of these possibilities. Also, recall that $N$ is a proper normal subgroup of $G$ and $G\leq {\rm Aut}(N)$.
 Now, if $N \cong {\rm Alt}(5)$, since $N<G$ then $G\cong {\rm Aut}({\rm Alt}(5))\cong {\rm Sym}(5)$, and we observe that ${\rm Sym}(5)$ satisfies the hypotheses of the theorem.
 An analogous situation happens with $N\cong {\rm PSL}_2(7)$ because  $G\cong {\rm Aut}({\rm PSL}_2(7))\cong {\rm PGL}_2(7)$ also verifies the hypotheses.
  Finally, if $N\cong {\rm PSL}_2(11)$ or ${\rm PSL}_3(3)$, then Out$(N)\cong C_2$ and $G\cong {\rm PGL}_2(11)$ or $G\cong {\rm PGL}_3(3)$, respectively.
  However, $ {\rm PGL}_2(11)$ has maximal subgroups isomorphic to ${\rm Sym}(4)$ that are neither supersolvable nor have prime or squared prime index, and similarly happens with $ {\rm PGL}_3(3)$, which has maximal subgroups isomorphic to ${\rm Sym}(4)\times C_2$. Thus, these cases can be eliminated, and therefore, we obtain two of the groups of (ii).

\medskip

Assume now that   $|G:M|=q^2$ with $q$ prime.
Given that $G=MN$,  if  we denote $N_0=M\cap N$, then by  applying Lemma \ref{m2}, we obtain $N_0=K_1\times \cdots\times K_n$,
where $K_i < S_i$ and $M$ acts transitively on the set  $\{K_1,\cdots, K_n\}$. This occurs because   Lemma \ref{m2}(2) would imply that $|S_i|$ divides $|G:M|$, which is not possible.
 It follows that $|G:M|=|N:N_0|=|S_i:K_i|^n=q^2$, forcing $n\leq 2$. Accordingly, we distinguish two subcases: $n=2$ and $n=1$.

\medskip
  (a) Assume first that $n=2$ and write $N_0=K_1\times K_2$. Given that $|S_i:K_i|=q$, we have $S_i/{\rm core}_{S_i}(K_i)\leq {\rm Sym}(q)$, and hence,
      $q$ does not divide $|K_i|$. Consequently,  $q$ is the largest prime divisor of $|S_i|$.
 We also have $N\lhd G\leq {\rm Aut}(N)$, and it is well known that ${\rm Aut}(N)$ is isomorphic to the standard wreath product ${\rm  Aut}(S_i) \wr {\rm Sym}(2)$ (see for instance \cite[3.3.20]{Robinson}),  that is,
  $G\leq ({\rm Aut}(S_1)\times {\rm Aut}(S_2))\rtimes A$, where $A$ permutes ${\rm Aut}(S_i)$ and is cyclic of order 2.
   Furthermore, notice that $G$ cannot be contained in ${\rm Aut}(S_1)\times {\rm Aut}(S_2)$, otherwise $N$ would not be minimal normal in $G$. This allows to assume without loss of generality that $A\leq G$.  Also,  since $|G:M|=q^2$ with $q$ odd, then $M$ certainly contains some Sylow $2$-subgroup of $G$, so $A^g\leq M$ for some $g\in G$.
    Thus,  we can also assume that $A\leq M$.
    On the other hand, as $S_i$ has a subgroup of  prime index and belongs to the list of simple groups given in Theorem \ref{2}, we combine these results identically as in (1.2) to obtain
     $S_i\cong {\rm Alt}(5)$, ${\rm PSL}_2(7)$, $ {\rm PSL}_2(11)$ or ${\rm PSL}_3(3)$. In the following, we analyze these cases, which have a very similar treatment. 
\medskip

(a.1) Suppose first $S_i\cong {\rm Alt}(5)$, so Aut$(S_i)\cong {\rm Sym}(5)$. Then
$$N= {\rm Alt}(5)\times {\rm Alt}(5) \lhd G\leq ({\rm Sym}(5)\times{\rm Sym}(5))\rtimes A.$$
 As we are assuming $A\leq G$, then there exist exactly three possibilities for $G$: either $G={\rm Alt}(5)\wr A$, or $G={\rm Sym}(5) \wr A$, or $G=NLA$ where $L=\{(d,d) \mid d\in {\rm Sym}(5)\}$ is the diagonal subgroup of ${\rm Sym}(5)\times {\rm Sym}(5)$.
 In the former case,  we consider  the diagonal subgroup  $D$ of $N$, which is  maximal in $N$ (by  Remark \ref{rd}, or instead, see \cite[Problem 7.7]{Isaacs}).
  This certainly implies that $D\times A$ is maximal in $G$. But this subgroup neither is supersolvable nor has prime power index,
    contradicting our hypotheses. Suppose now that $G={\rm Sym}(5)\wr A$ and consider a subgroup  $H$ of ${\rm Sym}(5)$ isomorphic to ${\rm Sym}(3)\times {\rm Sym}(2)$,
     which is maximal in ${\rm Sym}(5)$  \cite{Liebeck}. Then, it is not difficult to see that $H\wr A$ is maximal in $G$ (for instance, by means of \cite{GAP}).
 Alternatively, we can prove it through an application of Lemma \ref{m2}.  Write $H=H_0 \times B$ with $H_0\cong {\rm Sym}(3)\leq {\rm Alt}(5)$ and $B\cong C_2$ not contained in ${\rm Alt}(5)$. We know that $N$ is a minimal normal subgroup of $G$ and we have $(H\wr A)N=G$. Now, if $K$ is a maximal subgroup of $G$ containing $H\wr A$, then by Dedekind's rule, $K=(H \wr A)(K\cap N)$. Therefore, by applying Lemma \ref{m2}, we have two possibilities: either $K\cap N\cong {\rm Alt}(5)$, which is not possible due to order considerations because $|G:K|$ would not be a prime nor a squared prime; or $K\cap N=L\times L$, where $L<{\rm Alt}(5)$. However, the fact that $H_0\times H_0\leq K\cap N$ and that $H_0\cong {\rm Sym}(3)$ is maximal in ${\rm Alt}(5)$ imply that $H_0=L$. As a consequence, $K=H\wr A$,  proving the maximality of $H\wr A$ in $G$, as desired. Nevertheless,  the index of $H\wr A$ in $G$ is 100 and it is not supersolvable, contradicting the hypotheses. Finally, suppose that $G=NLA$. We claim that $LA$ is maximal in $G$. Let $H$ be a subgroup of $G$ with $LA<H<G$. Then $H= H\cap NLA=
(H\cap N)LA $ and hence  $L < L(H\cap N)< {\rm Sym}(5)\times {\rm Sym}(5)$. Since ${\rm Alt}(5)$ is the unique non-trivial proper normal subgroup of ${\rm Sym}(5)$, by using Remark \ref{rd} we obtain that $L(H\cap N)$ is the only proper subgroup of ${\rm Sym}(5)\times {\rm Sym} (5)$ distinct from $L$ that contains $L$. Such subgroup is necessarily
$LN$, so $L(H\cap N)= LN$. Consequently, $H=LNA=G$, a contradiction, which proves the claim.  Thus, $LA\cong {\rm Sym} (5)\times A$ is a maximal subgroup of $G$ that is not supersolvable nor  has prime index or squared prime index. Therefore, the  case $G=NLA$ is excluded too. 

\medskip
(a.2) If  $S_i \cong {\rm PSL}_2(11)$, write $N={\rm PSL}_2(11)\times {\rm PSL}_2(11)$. We have Out$(S_i)\cong C_2$, so again we get three possibilities: $G={\rm PSL}_2(11)\wr A$,  ${\rm PGL}_2(11)\wr A$, or $NLA$ where $L$ is the diagonal subgroup of ${\rm PGL}_2(11)\times {\rm PGL}_2(11)$.
 In order to rule out the former case, it suffices to consider the (maximal) diagonal subgroup $D$  of ${\rm PSL}_2(11)\times {\rm PSL}_2(11)$,
because then $D \times A$ is clearly maximal in $G$, but neither is supersolvable nor has prime power index.
For  discarding $G= {\rm PGL}_2(11)\wr A$, we may consider, for instance, a maximal subgroup $H\cong D_{24}$ of ${\rm PGL}_2(11)$. This subgroup may be written as $H=H_0B$ with $H_0\unlhd H$ and $B\cong C_2$, where $D_{12}\cong H_0$ is a maximal subgroup of ${\rm PSL}_2(11)$.
    A similar argument as in (a.1), or alternatively using \cite{GAP},  allows to demonstrate that $H\wr A$ is  maximal in $G$.
    It is evident that this subgroup  neither has  prime index nor  squared prime index. Furthermore,  $H\wr A$ is not supersolvable by Lemma \ref{wr}. Finally, when $G=NLA$,  taking into account that ${\rm PSL}_2(11)$ is the unique non-trivial proper subgroup of ${\rm PGL}_2(11)$, an analogous argument to that in (a.1),  can be made for proving that $LA$ is  maximal in $G$, and certainly that does not satisfy the hypotheses of the theorem. 

\medskip
(a.3) Assume $S_i\cong {\rm PSL}_2(7)$, write $N={\rm PSL}_2(7)\times {\rm PSL}_2(7)$. Then  Out$(S_i)\cong C_2$ and we have  either $G={\rm PSL}_2(7)\wr A$, ${\rm PGL}_2(7)\wr A$ or $NLA$ where $L$ denotes the diagonal subgroup of  ${\rm PGL}_2(7)\times {\rm PGL}_2(7)$.
 The  reasonings for discarding these groups are  identical to those employed  in (a.1) and (a.2).
  Specifically, for $G={\rm PSL}_2(7)\wr A$ we may consider the diagonal subgroup, and for $G= {\rm PGL}_2(7)\wr A$,
  it suffices to choose a maximal subgroup  $H$ of ${\rm PGL}_2(7)$ isomorphic to $C_7\rtimes C_6$, thereby obtaining  the maximal subgroup $H\wr A$  of $G$.
   This subgroup neither is supersolvable (by Theorem \ref{wr}) nor  has  prime power index. Finally, if $G=NLA$, since ${\rm PSL}_2(7)$ is the unique non-trivial proper subgroup of ${\rm PGL}_2(7)$, it turns out that $LA$ is a maximal non-supersolvable subgroup of $G$ whose index is neither prime nor squared prime.
   
\medskip 
(a.4)
Finally, suppose that $S_i\cong {\rm PSL}_3(3)$ and write $N={\rm PSL}_3(3)\times {\rm PSL}_3(3)$. Then Out$(S_i)\cong C_2$ and we have either $G={\rm PSL}_3(3)\wr A$,  ${\rm PGL}_3(3)\wr A$ or $NLA$ ,with  $L$ being the diagonal subgroup of  ${\rm PGL}_3(3)\times {\rm PGL}_3(3)$.
 Again, these groups are discarded by using identical arguments to those employed  in the above cases.
  For $G={\rm PSL}_3(3)\wr A$ we  consider the diagonal subgroup of $N$, giving rise to a maximal subroup of $G$ which does not satisfy the hypotheses. For $G= {\rm PGL}_3(3)\wr A$,
  we choose a maximal subgroup  $H$ of ${\rm PGL}_3(3)$ isomorphic to ${\rm Sym}(4) \rtimes C_2$, and get the maximal subgroup $H\wr A$  of $G$, which neither is supersolvable (by Theorem \ref{wr}) nor  has  prime power index. Finally, if $G=NLA$, again by taking into account that ${\rm PSL}_3(3)$ is the only non-trivial proper subgroup of ${\rm PGL}_3(3)$, we  deduce that $LA$ is a maximal non-supersolvable subgroup of $G$ whose index is neither prime nor squared prime.

\medskip
(b) Assume that $n=1$. This means that $N$ is non-abelian simple  and $N\lhd G\leq {\rm Aut}(N)$, with  $|G:M|=|N:N\cap M|=q^2$.
In view of  Theorem \ref{clas}  and  the fact that $N$ belongs to the list appearing in Theorem \ref{2}, it is sufficient to consider $N\cong{\rm PSL}_2(8)$.
 In this case, Out$(N)\cong C_3$ and $G\cong {\rm P\Gamma L}_2(8)$, which is easily  verified to meet the hypotheses of the theorem.
 This group corresponds to the (third) remaining group of statement (ii).

 \medskip
(2.2)  Suppose that each maximal subgroup of $G$ that does not contain $N$ is supersolvable.
   First, we claim that the normalizer in $G$ of every Sylow subgroup of $N$ is supersolvable, so  in particular, the normalizer in $S_i$ of every Sylow subgroup of $S_i$ must be supersolvable as well.
   Indeed, let $P$ be a Sylow $p$-subgroup of $N$, where $p$ is an arbitrary prime dividing $|N|$.
   By the  Frattini argument,   $G={\bf N}_G(P)N$, and then there exists a maximal subgroup of $G$ that contains ${\bf N}_G(P)$ and does  not contain $N$.
   Consequently, our above assumption  implies that such a maximal subgroup is supersolvable.
   In particular, ${\bf N}_G(P)$ and ${\bf N}_{N}(P)=\prod{\bf N}_{S_i}(P_i)$  where $P_i=P\cap S_i\in {\rm Syl}_p(S_i)$ are both supersolvable, so the assertion is proved.

\medskip
   Below we present a  case-by-case analysis of the distinct possibilities for $S_i$ appearing in Theorem \ref{2}.

\medskip
 We start our analysis by assuming  $S_i\cong {\rm  PSL}_2(2^p)$ with $p$ prime, and consider the prime $2$.
  Lemma \ref{lemma2}(1) asserts that if $P_i\in {\rm Syl}_2(S_i)$, then ${\bf N}_{S_i}(P_i)\cong P_i\rtimes C_{2^p-1}$, which is not supersolvable.
   This contradicts the aforementioned property, so this case can be dismissed.
   A similar  reasoning can be employed to reject the Suzuki group, ${\rm Sz}(q)$, with $q=2^{2n+1}$,
    given that the normalizer of a Sylow $2$-subgroup $P$ is isomorphic to $P\rtimes C_{q-1}$ \cite[Chap XI, Theorem 3.10]{HB}.

 \medskip

   For the group $S_i\cong {\rm PSL}_3(3)$,  we  should note that all Sylow normalizers are  supersolvable. This can be easily checked with the help of \cite{GAP}, so the above argument cannot be applied.
   However,  we can instead check (e.g. in \cite{Liebeck2}) that $S_i$ has a unique conjugacy class of maximal subgroups  isomorphic to ${\rm Sym}(4)$.
   Then the direct product of $n$ such subgroups, each one lying in a distinct $S_i$,  constitute a single conjugacy class of subgroups of $N$.
   Thus, take $H_0\cong {\rm Sym}(4)$ a maximal subgroup  of $S_i$  and let $H=H_0 \times\cdots\times H_0\leq N$.
  Next, we see that the Frattini argument can be applied so as to get $G={\bf N}_G(H)N$.
   Indeed, if $g\in G$, since $G$ (transitively) permutes the factors $S_i$, then  $H^g=\prod H_0^g$. Furthermore,
   each factor belongs to a distinct $S_i$, is maximal in such $S_i$ and isomorphic to ${\rm Sym}(4)$.
    Accordingly, $H^g$ is the direct product of $n$  maximal subgroups  isomorphic to ${\rm Sym}(4)$, so it belongs to the above-mentioned conjugacy class of subgroups of $N$.
    It follows that $H^g=H^n$ for some $n\in N$, and this yields to $G={\bf N}_G(H)N$, as required.
    Now, if we take a maximal subgroup $M$ of $G$ containing ${\bf N}_G(H)<G$, it is clear that $N\not\leq M$ and thus, by  our assumption, $M$ is supersolvable.
    However, this is not the case for ${\bf N}_G(H)$, as $H$ is not supersolvable, providing a contradiction. The conclusion is that this case cannot occur.

\medskip
Suppose now that $S_i\cong {\rm PSL}_2(3^r)$ with $r$ prime. If $r$ is odd, then  $3^r\equiv 3$ (mod $8$).
 Moreover,  if  $P_i\in {\rm Syl}_2(S_i)$,   then  ${\bf N}_{S_i}(P_i)\cong {\rm Alt}(4)$  by Lemma \ref{lemma2}(4.1).
 As ${\rm Alt}(4)$  is not supersolvable, these cases are excluded too.
  Therefore, the only remaining possibility within this family is when $r=2$, that is, $S_i\cong {\rm PSL}_2(3^2)$. But this group also belongs to  case (3) of Theorem \ref{2}, so it will be discussed below.

  \medskip
     Finally, assume that $S\cong {\rm PSL}_2(p^{2^a})$ with $p$ an odd prime and $a\geq 0$.
      The  normalizers of Sylow $2$-subgroups of this group are supersolvable except when $p^{2^a}\equiv \pm 3$ (mod $8$) (see again Lemma \ref{lemma2}),
      when they are  isomorphic to ${\rm Alt}(4)$, which is not supersolvable.
      But the above  congruence only occurs  when $a=0$ and $p\equiv \pm 3$ (mod  $8$), because if $a\geq 1$, then $p^{2^a}-1=(p^{2^{a-1}}+1)(p^{2^{a-1}}-1)$ with
      $p$ and odd prime is the product of two consecutive even numbers,   and hence $p^{2^a}\equiv 1$ (mod $8$).
        Consequently,  within this family of simple groups, only the case $a=0$ and  $p\equiv \pm 3$  (mod  $8$) can be excluded.

\medskip
       This finishes our discussion on simple groups, so we have shown that $S_i$ can only be isomorphic to one of the simple groups listed in (iii) of the statement of the theorem.
 The rest of the proof is concerned with proving that $G/N$ is a $2$-group.

 \medskip
 Put $S=S_1$. We know that $N\lhd  G \leq {\rm Aut}(N)\cong  {\rm Aut}(S)\wr {\rm Sym}(n)$ \cite[3.3.20]{Robinson},
 where   ${\rm Aut}(S) \wr {\rm Sym}(n)$ denotes the standard wreath product.
  Write $A={\rm Aut}(S)$ and  let $A^*$ be the base group of $A \wr {\rm Sym}(n)$.
  In order to prove that $G/N$ is a $2$-group we will show that $(G\cap A^*)/N$ and  $G/(G\cap A^*)$  are $2$-groups.
  For the former group, we know that
  $$(G\cap A^*)/N\leq {\rm Out}(S)\times\cdots\times {\rm Out}(S).$$
   As we have proved above that $S\cong {\rm PSL}_2(p^{2^a})$ with $p$ an odd prime, then  Out$(S)$ has order $(2, p^{2^a}-1) 2^a= 2^{a+1}$ \cite{Con},
    so our first assertion immediately follows. We prove below that $G/(G\cap A^*$) is a $2$-group as well.
      Notice that $$G/(G\cap A^*)\cong A^*G/A^*=(A^*G\cap {\rm Sym}(n))A^*/A^*\cong A^*G\cap {\rm Sym}(n).$$
We suppose, on the contrary, that there is a prime $r\neq 2$ dividing $|A^*G\cap {\rm Sym}(n)|$ and seek a contradiction.
Let $P_0\in {\rm Syl}_2(S)$ and write $P=P_0\times \cdots \times P_0\in {\rm Syl}_2(N)$. The Frattini argument gives $G={\bf N}_G(P)N$.
 Now, since $|G/(G\cap A^*)|$ is divisible by $r$ and $N\leq G\cap A^*$, one easily deduces that there exists an $r$-element $x\in G\setminus G\cap A^*$ such that $x\in {\bf N}_G(P)$.
 Moreover, we can write $x= as$, with $a\in A^*$ and $1\neq s\in {\rm Sym}(n)$.
 It is straightforward that $s$ normalizes $P$, and  hence, $a\in {\bf N}_{A^*}(P)={\bf N}_A(P_0)\times \cdots \times {\bf N}_A(P_0)$.
 Since $s$ permutes non-trivially the direct factors of $P$, it is immediate that $x$ does not centralize $P$.
 But we have shown above that  ${\bf N}_G(P)$ is supersolvable, so in particular, it possesses a normal $2$-complement. This forces  $x$ to centralize $P$, thereby providing a contradiction.
 This shows that $G/(G\cap A^*)$ is a $2$-group, as required.
 As a result,  ${\bf O}^2(G)=N$.
 Finally, since we are assuming $N<G$, we conclude that $G$ meets  all the conditions of (iii), which ends the proof.
\end{proof}

\begin{proof}[Proof of Corollary B]
It is enough to apply  Theorem A, simply considering that if $G$ is a non-abelian simple group, then $S(G)=1$ and ${\bf O}^2 (G)=G$. In fact, the proof of this corollary corresponds with the proof of case 1 in Theorem A.
\end{proof}

\noindent
{\bf Acknowledgements}
This work is supported by the National Nature Science Fund of China (No. 12471017 and No. 12071181) and the  first named author is  also supported by Generalitat Valenciana,
 Proyecto CIAICO/2021/193. The  second  named author is also supported by the  Natural Science Research Start-up Foundation of Recruiting Talents of Nanjing University of Posts and Telecommunications
  (Grant Nos. NY222090, NY222091).

\medskip
\noindent
{\bf Data availability} Data sharing not applicable to this article as no data sets were generated or analyzed during
the current study.

\bigskip
\noindent
{\bf \large Declarations}

\medskip
\noindent
{\bf Conflict of interest} The authors have no conflicts of interest to declare.

\bibliographystyle{plain}

\begin{thebibliography}{1}



\bibitem{Ballester}
Ballester-Bolinches, A., Esteban-Romero, R., Lu, Jiakuan.:
On finite groups with many supersolvable subgroups
\newblock{\em Arch. Math. (Basel)}  {\bf 109}, no 1, 3-8 (2017).       https://doi.org/10.1007/s00013-017-1041-4 

\bibitem{Ballester2}
Ballester-Bolinches, A., Cossey, J., Esteban-Romero, R.:
On the abnormal structure of finite groups
\newblock{\em Rev. Mat. Iberoam.} {\bf 30}  no 1. 13-24 (2014).    https://doi.org/10.4171/rmi/767


\bibitem{BS}
Beltr\'an, A., Shao, C.:
Extensions of a theorem of P. Hall on indexes of maximal subgroups. Accepted in {\em Forum Math.} 	arXiv:2501.02249

\bibitem{Con}
Conway J.H., Curtis,  R.T., Norton, S.P., Parker, R.A. and Wilson, R.A.:
\newblock  Atlas of Finite Groups.
 \newblock  Oxford Univ. Press, London, (1985).

 \bibitem{DM}
Demina, E.N., Maslova, N.V.: \newblock  Non-abelian composition factors of a finite group with arithmetic constraints on non-solvable maximal subgroups,
 \newblock{\em Tr. Inst. Mat. Mekh.} {\bf 20}, No. 2, 122-134 (2014); translation in {\em Proc. Steklov Inst. Math.} 289 (2015) S64-S76.    https://doi.org/10.1134/S0081543815050065   


\bibitem{Doerk}
Doerk, K.: Minimal nicht \"uberaufl\"osbare, endliche Gruppen. Math. Z. {\bf 91}
(1966), 198-205.


\bibitem{Durbin}
Durbin, J.R.:
Finite supersolvable wreath products.
{\em Proc. Amer. Math. Soc.} {\bf 17}(1): 215-218 (1966).



\bibitem{Guo}
Guo, W.,  Kondrat'ev, A.S.,  Maslova, N.V., Myao, L.:
 Finite groups whose maximal subgroups are solvable or have prime power indices.
 {\em Tr. Inst. Mat. Mekh.}, 2020, {\bf 26} no 2: 125-131; translation in {\em Proc. Steklov Inst. Math.},  309, S47-S51, (2020).   https://doi.org/10.1134/S0081543820040069 

\bibitem{Gura}
Guralnick, R.M.:
\newblock  Subgroups of prime power index in a simple group.
\newblock {\em J. Algebra}  {\bf 81}(2): 304-311 (1983).


 \bibitem{Hup}
Huppert, B.:
 \newblock Endliche Gruppen I.
 \newblock Springer, Berlin, (1967).

\bibitem{HB}
Huppert, B., Blackburn, H.:
\newblock  Finite Groups II-III.
Springer, Berlin, (1982).

\bibitem{Isaacs}
Isaacs, I.M: Algebra: A graduate course. Graduate texts in Mathematics, vol. 100. American Mathematical Society (1994). 

\bibitem{Liebeck}
Liebeck, M.W., Praeger, C.E., Saxl, J.: A classification of the maximal subgroups of the finite alternating and symmetric groups. {\em J. Algebra} {\bf 111},  365-383 (1987).

\bibitem{Liebeck2}
Kleidman, P.,  Liebeck, M.W.: The Subgroup Structure of the Finite Classical Groups, vol. 129. London
Math. Soc. Lecture Note Ser., Cambridge University Press, Cambridge (1990).

 \bibitem{LPZ}
 Lu, J., Pang, L., Zhong, X.: Finite groups with non-nilpotent maximal subgroups. {\em Monatsh. Math.} {\bf 171} (2013) 425-431.  https://doi.org/10.1007/s00605-012-0432-7  


\bibitem{Robinson}
 Robinson, D.J.:
 A course in the theory of groups, 2nd ed., Springer-Verlag, New
York-Heidelberg-Berlin, (1996).


\bibitem{SLT}
Shi, J., Liu, W., Tian, Y.:
A note on finite groups in which every non-nilpotent maximal subgroup has prime index.
{\em J. Algebra Appl.} {\bf 23}, 9, (2024). https://doi.org/10.1142/S0219498824501354


\bibitem{GAP} The GAP Group, GAP - Groups, Algorithms and
Programming, Vers. 4.12.2, 2022.

 (http://www.gap-system.org)

\bibitem{Q}
Qian, G.H.:  Non-solvable groups with few primitive character degrees, {\em J. Group Theory} {\bf 21},  2,  295-318, (2018). https://doi.org/10.1515/jgth-2017-0037
 

\bibitem{YJK}
 Yi, X., Jian, S., Kamornikov, F.:
Finite groups with given non-nilpotent maximal subgroups of prime index.
{\em J. Algebra Appl. } {\bf 18}, 5 (2019). https://doi.org/10.1142/S0219498819500877






\end{thebibliography}

\end{document}